\documentclass[12pt,a4paper,oneside]{amsart}

\usepackage{amsfonts, amsmath, amssymb, amsthm, amscd}
\usepackage{anysize}
\marginsize{2cm}{2cm}{1.5cm}{1.5cm}

\usepackage[T2A]{fontenc}
\usepackage[utf8]{inputenc}
\usepackage[english]{babel}
\usepackage{cmap}
\usepackage{enumerate}

\usepackage{xcolor}
\usepackage{hyperref}
\hypersetup{
	colorlinks   = true, 
	urlcolor     = blue, 
	linkcolor    = blue, 
	citecolor    = red 
}
\newcommand{\Href}[2]{\hyperref[#2]{#1~\ref{#2}}}
\usepackage{anysize}
\marginsize{2cm}{2cm}{1.5cm}{1.5cm}

\usepackage{cancel}

\newtheorem{thm}{Theorem}

\newtheorem{lemma}{Lemma}[section]
\newtheorem{cor}{Corollary}[section]

\theoremstyle{definition}

\newcommand{\norm}[1]{\left\|#1\right\|}

\newcommand{\EEop}{\mathop{\mathbb E}}

\newcommand{\conv}{\mathrm{conv}}

\def\R{{\mathbb R}}

\def\phi{\varphi}
\def\epsilon{\varepsilon}

\def\dist{\mathrm{dist}}

\DeclareMathOperator{\sgn}{sgn}

\def\co{\mathop{\rm co}}

\def\diam{\mathop{\rm diam}}

\newcommand{\prooff}{\noindent {\bf Proof.}\\}%
\newcommand{\bbox}{\par\noindent\ensuremath{\Box}\par\noindent}%

\newcommand{\centroid}[1]{\operatorname{c} \left(#1 \right)}

\newcommand{\eps}{\epsilon}
\newcommand{\CE}{\operatorname{Cext}}
\newcommand{\radem}[1]{\operatorname{R}\!\left( #1 \right)}%

\title{No-dimension Tverberg's theorem and its corollaries in Banach spaces of type $p.$}

\author{Grigory Ivanov}
\address{
Institute of Discrete Mathematics and 
Geometry\\ TU Wien\\ Wiedner Hauptstraße 8-10/104\\
A-1040 Wien\\ Austria}
\address{Department of Higher Mathematics, Moscow Institute of Physics and Technology,  Institutskii pereulok 9, Dolgoprudny, Moscow
region, 141700, Russia.}
\email{grimivanov@gmail.com}
\thanks{Supported by the Swiss National Science Foundation grant 200021\_179133. Supported by Russian Foundation for Basic Research, project 18-01-00036 A}

\begin{document}
\maketitle

{\bf Abstract.} We continue our study of  'no-dimension'   analogues of  basic theorems in combinatorial and convex geometry  in  Banach spaces. We generalize some results  of the paper  \cite{adiprasito2019theorems} and  prove 
no-dimension versions of colorful Tverberg's theorem, selection lemma and the weak $\eps$-net theorem  in  Banach spaces of type $p > 1.$ To prove this results we use the original ideas of 
\cite{adiprasito2019theorems} for the Euclidean case and  our slightly modified version of the celebrated Maurey lemma.
\bigskip



\section{Introduction}  

In \cite{adiprasito2019theorems}, the authors started  a systematic study of 
what they called   ‘no-dimension’ analogues of  basic theorems in combinatorial and convex geometry 
such as Carath{\'e}odory's, Helly's and Tverberg's theorems and others. All original versions of these theorems state different combinatorial properties of convex sets  in $\R^d.$ 
And the results  depend on the dimension $d$ (some of these theorems can be used to characterize the dimension). The idea behind these ‘no-dimension’  or approximate versions of the well-known theorems is to make them independent of the dimension.
However, it comes at some cost -- the approximation error.
For example, Carath{\'e}odory's theorem states that any point $p$ in the convex hull of a set $S \in \R^d$  is a convex combination of at most $d+1$ points of $S.$ In Theorem 2.2 of  \cite{adiprasito2019theorems} the authors proved that the distance between  any point $p$ in the convex hull of a bounded set $S$ of a Euclidean space  and  the $k$-convex hull  is at most $\frac{\diam S}{\sqrt{2k}}$. Here, the $k$-convex hull of $S,$ is the set of all convex combinations of at most $k$ points of $S.$ Clearly, the last statement doesn't involve the dimension, but it can guarantee only an approximation of a point.

All the proofs in \cite{adiprasito2019theorems} exploit   the properties of Euclidean metric significantly. And in general, this type of questions were mostly considered in the Euclidean case
(see, for example, the survey \cite{fradelizi2017convexification} and the references herein). 
Probably, there is only one exception at the moment. The celebrated Maurey lemma \cite{pisier1980remarques} is  an approximate version of Carath{\'e}odory's theorem   for Banach spaces that have (Rademacher) type $p > 1.$ Recently Barman \cite{barman2015approximating} showed that  Maurey's lemma is useful for some algorithms (for example, for computing Nash equilibria and for densest bipartite subgraph
problem). Moreover, different problems about  approximation of operators (see \cite{FY17}, \cite{ivanov2019approximation}) can be reformulated in the language of no-dimension theorems.

In this paper we continue our study of no-dimension theorems in Banach spaces started in \cite{ivanov_approx_car_19}, where the author provided  a greedy algorithm proof of Maurey's lemma in a uniformly smooth Banach space.
The main results of this paper is the generalization of approximate Tverberg's theorem and its corollaries to Banach spaces of type $p.$

A Banach space $X$ is said to be of {\it type} $p$ for some $1 < p \leqslant 2,$ if there exists a constant $T_p(X) < \infty$ so that, for every finite set of vectors $\left\{x_{j}\right\}_{j=1}^{n}$ 
in $X,$ we have 
$$
\int_{0}^{1}\left\|\sum_{j=1}^{n} R_{j}(t) x_{j}\right\| d t \leqslant T_p(X)\left(\sum_{j=1}^{n}\left\|x_{j}\right\|^{p}\right)^{1 / p},
$$
where  $\left\{R_{j}\right\}_{j=1}^{\infty}$ denotes the sequence of the Rademacher functions.

Throughout the paper, $X$ is a Banach space of type $p,$ $ 1 < p  \le 2,$  $T_p(X)$ denotes the constant that appears in the definition of type and  $w = \frac{1-p}{p}.$
We use $C(X)$ to denote a universal constant for a given space $X.$ 

The following statement are the main results of the paper.

\begin{thm}[No-dimension colorful Tverberg theorem]\label{thm:col_nodim_Tv}
Let $Z_1, \dots, Z_r \subset X$ be $r$ pairwise-disjoint sets of points in a Banach space $X$ and with $|Z_i| = k$ for all $i \in [r].$ Let $S = \bigcup_1^r Z_i$ and $D = \max\limits_{i} \diam Z_i.$ Then there is a point $q$ and a partition $S_1, \dots, S_k$ of $S$ such that $|S_i \cap Z_j| = 1$ for every $i \in [k]$ and every $j \in [r]$ satisfying
\begin{equation}
\dist \left( q, \conv{S_i}\right) \le  C(X) r^{w} D \text{  for every } i \in [k].
\end{equation} 
It is enough to set $C(X) = \frac{2^{1/p}}{1 - 2^w} T_p(X).$
\end{thm}

\begin{thm}[No-dimension selection lemma]\label{th:selection} Given a set $P$ in a Banach space $X$ with $|P|=n$ and $D=\diam P$ and an integer $r \in [n]$, there is a point $q$ such that the ball
 $B\left(q,C(X) r^{w} D\right)$ intersects the convex hull of $r^{-r} {n \choose r}$ $r$-tuples in $P$. It is enough to set $C(X) =  2 \frac{2^{1/p}}{1 - 2^w} T_p(X).$
\end{thm}

\begin{thm}[No-dimension  weak $\eps$-net theorem]\label{th:epsnet} Assume $P$ is a subset of a Banach space $X,$ $|P|=n$, $D=\diam P$, $r\in [n]$ and $\eps >0$. Then there is a set $F \subset X$ of size at most $r^r \eps^{-r}$ such that for every $Y\subset P$ with $|Y| \ge \eps n$
\[
\left(F + B\left(0,C(X) r^{w} D\right)\right)\cap \conv Y \ne \emptyset.
\]
It is enough to set $C(X) =  2 \frac{2^{1/p}}{1 - 2^w} T_p(X).$
\end{thm}

In fact, we just generalize the averaging technique used in \cite{adiprasito2019theorems} to prove 
no-dimension versions of colorful Tverberg's theorem, selection lemma and the weak $\eps$-net theorem for the Euclidean space. For this purpose, we prove the following refinement of the celebrated Maurey lemma.

 For a positive integer $k$, we use the notation $[k] = \{1, \dots , k\}$, and we use  $\binom{S}{k}$ to denote the set of the $k$-element subsets of $S.$   Given a finite set $S$ in a linear space, we denote by $\centroid{S}$ the centroid of $S,$ that is,
 \[
 \centroid{S} = \frac{1}{|S|} \sum\limits_{s \in S} s.
 \]

\begin{thm}\label{thm:Maurey_for_centroid}
Let $S $ be  a  set of $n \ge 2$ pairwise different points in a Banach space $X.$  
Then there is a partition of $S$ into two sets $S_1, S_2$ of size $\lfloor \frac{n}{2}\rfloor$ and 
$\lceil \frac{n}{2}\rceil$ such that 
\[
\norm{c(S_1) - c(S_2)} \le  C(X) \,  \left\lceil\frac{n}{2} \right\rceil^w \diam S.
\]
It is enough to set $C(X) =  2^{1/p} T_p(X).$
\end{thm}

Using \Href{Theorem}{thm:Maurey_for_centroid}, our proofs of \Href{Theorem}{thm:col_nodim_Tv}, \Href{Theorem}{th:selection} and 
 \Href{Theorem}{th:epsnet} follow the same lines as  in \cite{adiprasito2019theorems}.
We tried to be as close to the original proofs in the Euclidean case as possible. 
Our results give the same asymptotic in $r$ as in \cite{adiprasito2019theorems} for the Euclidean case and constants are reasonably close (our constant in \Href{Theorem}{th:selection} and 
 \Href{Theorem}{th:epsnet} is $2 \sqrt 2 (\sqrt{2} + 1),$ which is  twice the constant obtained in 
 \cite{adiprasito2019theorems} for the Euclidean case).

In the next Section, we discuss the Maurey lemma and prove \Href{Theorem}{thm:Maurey_for_centroid}.
Then, in \Href{Section}{sec:proofs_main_result},  we prove the main results of the paper.

\section{Averaging technique}
\subsection{Maurey's lemma}
Maurey’s lemma \cite{pisier1980remarques} is  an approximate version of the  Carath{\'e}odory theorem for Banach spaces of  type $p > 1.$ 
We can formulate Maurey's lemma as follows (see also Lemma D in \cite{bourgain1989duality}).

{\it 
Let $S$ be a bounded set in a  Banach space $X$ and $a \in \co S.$
Then there exists a sequence $\{x_i\}_1^k \subset S$ such that for  vectors
$a_k = \frac{1}{k} \sum\limits_{i \in [k]} x_i$ the following inequality holds
\begin{equation*} 
\norm{a - a_k} \leq T_p (X) k^w \diam S.
\end{equation*}
}

As we need a slightly more general colored statement, we provide the proof of the following version of Maurey's lemma, which is trivially follows from the original one. 

We prefer to use a different notation for the averaging by the Rademacher variables.
For a finite set $S$ in a linear space $L,$ let $\radem{S}$ denote a random variable 
$\{-1,1\}^{|S|} \to L$ which takes $\pm s_1, \dots, \pm s_{|S|}$  with probability $1/2^{|S|}.$ That is, $\radem{S}$ is a choice of signs in the  sum  $\pm s_1, \dots, \pm s_{|S|}.$ By $\EEop\limits_{\mathbf{rad}}$ we denote the expected value of $\radem{S}.$ 
\begin{lemma}\label{le:CalCar}
Let $P_1, \ldots, P_r$ be $r$ point sets in $X$, $D = \max_{i\in [r]} \diam P_i$ and $\eta >0$. Assume that $B(a,\eta )\cap \conv P_i\ne \emptyset$ for every $i \in [r]$. 
Then there exist $r$ sequences $\{x^i_j\}_{j=1}^k \subset P_i$ such that for  vectors
$a^i_k = \frac{1}{k} \sum\limits_{j \in [k]} x^i_j$ the following inequality holds
$$\dist (a, \centroid{a^1_k, \dots, a^r_k}) \le  T_P(X)k^w r^w D + \eta.$$
\end{lemma}
\begin{proof}
Let $x_i$ be a point of $B(a,\eta )\cap \conv P_i,$ $ i \in [r].$ By the triangle inequality, 
it is enough to show that there is a proper multiset of points $T$ such that  
$$\norm{\centroid{T} - \centroid{\{x_1, \dots, x_r\}}} \le  T_p(X) k^w r^w D.$$
Since $x_i \in \conv P_i,$ there exist $y^i_1, \dots, y^i_{N_i} \in P_i$ and positive scalars
$\lambda^i_1, \dots, \lambda^i_{N_i},$ $\sum\limits_{j \in N_i} \lambda^i_j,$ such that 
$x_i = \sum\limits_{j \in N_i} \lambda^i_j y^i_j.$  Let $F_i$ be a $P_i$-valued random variable which takes $y^i_j$ with probability $\lambda^i_j.$ Let $F_i(1), \dots, F_i(k)$ and 
$F'_i(1), \dots, F'_i(k)$ be a series of  independent copies of $F_i,$ 
$i \in [r].$  Then
\[
\EEop\limits_{F_1, \dots, F_r} \norm{\sum\limits_{j \in [k]}\sum\limits_{i \in [r]} (F_i(j) - p_i)} \stackrel{\text{(Ave)}}{\le}  
\EEop\limits_{F_1, \dots, F_r} \EEop\limits_{F'_1, \dots, F'_r} \norm{\sum\limits_{j \in [k]}\sum\limits_{i \in [r]} (F_i(j) - F'_i(j))} \stackrel{\text{(S)}}{=}  
\]
\[
\EEop\limits_{F_1, \dots, F_r}  \EEop\limits_{\mathbf{rad}}\norm{ 
\radem{\left\{F_i(j) - F'_i(j)\right\}_{i \in [r]}^{j \in [k]}}} \le 
T_p(X) k^{1/p}r^{1/p} D,
\]
where in step (Ave) we use identity $\EEop\limits_{F_i} \left(F_i  - x_i\right) = 0,$  in step
 (S) we use  that  $(F_i(j) - F'_i(j))$'s are symmetric  and independent of each other, and the last inequality is a direct consequence of the definition of type $p$. 
 Dividing by $kr,$ we obtain the needed bound and complete the proof.
\end{proof}

In other words, we approximate a point in the convex hull by the centroid of a multiset, that is, one element might be counted several times.  But in our proofs we need  to choose different points or a subset of cardinality $k.$ And, as shows the following simple example, any point of the convex hull  cannot be approximated  by the centroids of  subsets which cardinalities are forced to be large. 

{\it \noindent Example.} Let points of a set $S$ are 'concentrated' around a unit vector $p.$
Adding $-p$ to $S,$ we may assume that $0 \in \conv S.$ It is easy to see, that  the centroid of any 
$Q \in \binom{S}{k}$ is at constance distance from the origin for $k \ge 4.$ 
 
However, as will be shown in the next Section, we always can approximate the centroid of the initial set by the centroids of its $k$-element subset.

\subsection{Approximation by the centroids}
Let $S$ is the disjoint  union of sets (considered colors) 
$Z_1, \dots, Z_r $, and each $Z_j$ has size $n \ge 2.$  
For any subset $Q$ of $S$ we use $Q_i$ to denote $Q \cap Z_i.$ 
Let $d \in [n-1].$  The set  of  all $(d \times r)$-element subsets $Q$ of $S$ such that $|Q_i| = d$ is denoted as $\binom{S}{d/r}.$ 
We use $(\Omega_d (S), p)$ to   denote a probability space on $\binom{S}{d/r}$  with the uniform distribution.  That is, the probability of choosing $Q \in \binom{S}{d/r}$ is 
\[
	\frac{1}{P_{(n,r)}^{d}},   \text{  where  } P_{(n,r)}^{d} = \binom{n}{d}^r .
\] 
We use $\sigma(S)$ to denote the sum of all elements of a set $S.$

The following statement is a direct corollary of  Jensen's inequality.

\begin{lemma}\label{le:jensen_ineq}
Under the above conditions, let additionally 
$S = \bigcup\limits_{i \in [r]} Z_i$ be  a subset of points in a linear space $L$  with  $\sigma(Z_i) = 0$ for all $i \in [r].$
Let $\phi$ be a convex function $L \to \R .$
Then
\begin{equation}\label{eq:lem_jensen_ineq}
   \EEop\limits_{\left(\Omega_d(S),p \right)} \phi(\gamma \sigma(Q)) \le   \EEop\limits_{\mathbf{rad}} \phi(\radem{S}),
\end{equation}
where $\gamma$ is real number from interval $(1,2)$ that depends only on $n$ and $d.$
\end{lemma}
\begin{proof} 
Firstly,  we explain the idea of the proof and then proceed with the technical details.
For a fixed $Q \in \binom{S}{d/r},$ we group all summands in the right-hand side of \eqref{eq:lem_jensen_ineq}  such  that $\radem{S}$ can be obtained 
from $\sigma(Q) - \sigma(S \setminus Q)$ by changing some signs either in the set $Q_i$ or in $Z_i \setminus Q_i$ for each $i \in [r].$ Then  we apply Jensen's inequality  for every  group of summands that corresponds to $Q \in \binom{S}{d/r}.$ Due to the symmetry,
 we understand that the argument of  $\phi$  looks like $a \sigma(Q) - b \sigma(S \setminus Q).$ Calculating the coefficients and using identity $\sigma(Q) = - \sigma(S \setminus Q),$ we prove the lemma.
 
 For  the sake of convenience, let $Q_i(+)  = Q_i$ and $Q_i(-) = Z_i \setminus Q_i,$
 $d(+) = d$ and $d(-) = n-d.$  Let $\CE_i (Q_i)$ be the set of all subsets $Y_i$  of $Z_i$ such that either 
 $Y_i$ is a non-empty subset of  $Q_i(+)$ or $Y_i \subset Q_i(-).$   We use $\CE(Q)$ to denote the set of all sets $Y \subset S$ such that $Y_i \in \CE_i (Q_i).$ For every $Y \in \CE(Q)$ we define $\sgn Y_i$ to be $+$
 or $-$ whenever $Y_i \subset Q_i(+)$ or $Y_i \subset Q_i (-),$ respectively.
 Let
 \[
 W(Y) = \prod\limits_{i \in [r]} \binom{d\left(-\sgn(Y_i)\right) +|Y_i|}{|Y_i|},
 \]     
be the weight of $Y \in \CE(Q).$

 Every signed sum  $\radem{S}$ can be represented as  
 $$\sigma(Q) - \sigma(S \setminus Q) - 2 \sum\limits_{i \in [r]} \sgn{Y_i} \cdot \sigma(Y_i),$$
 for some $Q \in \binom{S}{d/r}$ and $Y \in \CE{Q}.$   We see that $\sgn Y_i,$ $|Y_i|$ are determined by $\radem{S}$ in a unique way. Indeed, let $d_i$ be the number of pluses in the $i$-color of $\radem{S}.$  Then we need to swap exactly $d_i - d (+)$ of them to obtain a sum with exactly $d(+)$ pluses if $d_i  \ge d(+).$ If $d_i  < d(+),$ we need to swap exactly $d(+) - d_i = (n - d_i) - d(-)$ minuses.
 This implies that  the number  of such representations for a given $\radem{S}$ is $W(Y).$
 Therefore, we get
 \begin{equation}\label{eq:radem_representation}
  \EEop\limits_{\mathbf{rad}} \phi(\radem{S}) = \frac{1}{2^{nr}} \sum\limits_{Q \in \binom{S}{d/r}} \sum\limits_{Y \in \CE(Q)}
  \frac{1}{W(Y)} 
  \phi\!\!\left(\sigma(Q) - \sigma(S \setminus Q) - 2 \sum\limits_{i \in [r]} \sgn{Y_i} \cdot \sigma\left(Y_i\right) \right)
\end{equation}

By the symmetry and since the sum of coefficients at  $\phi$ is one, 
we have  
\begin{equation}\label{eq:double_count_total0}
\frac{1}{2^{nr}} \sum\limits_{Y \in \CE(Q)}
  \frac{1}{W(Y)}  = \frac{1}{P^d_{(n,r)}},
\end{equation}
for a fix $Q \in \binom{S}{d/r}.$ 

Using this and Jensen's inequality for each $Q \in \binom{S}{d/r}$ in the right-hand side of \eqref{eq:radem_representation}, we get that
\[
  \EEop\limits_{\mathbf{rad}} \phi(\radem{S}) \ge \frac{1}{P^d_{(n,r)}}  \sum\limits_{Q \in \binom{S}{d/r}} \phi\!\left(\sum\limits_{Y \in \CE(Q)}
  \frac{P^d_{(n,r)}}{2^{nr} W(Y)} \left[
  \sigma(Q) - \sigma(S \setminus Q) - 2 \sum\limits_{i \in [r]} \sgn{Y_i} \cdot \sigma\left(Y_i\right) \right] \right)
\] 

Let us carefully calculate the argument of $\phi$ in the right-hand side of the last inequality.
Due to the symmetry, this argument is 
$ \alpha^+ \sum\limits_{i \in [r]} \sigma(Q_i (+)) - \alpha^{-} \sum\limits_{i \in [r]}\sigma(Q_i (-)),$ where $\alpha^+$ and $\alpha^-$ are  some coefficients. Hence, it is enough to calculate the coefficients at $\sigma(Q_1(\pm) )$ only. 

Similarly to \eqref{eq:double_count_total0}, since the signs at the elements of different colors are chosen independently, we have
\begin{equation}\label{eq:double_count_total}
\frac{1}{2^{n(r-1)}} \sum\limits_{Y \in \CE(Q); \atop Y_1  \text{  is fixed}}
  \frac{W_1(Y)}{W(Y)}  = \frac{1}{P^d_{(n,r -1)}},
\end{equation}
where $W_1(Y) = \binom{d\left(-\sgn(Y_1)\right) +|Y_1|}{|Y_1|}.$

By \eqref{eq:double_count_total0} and \eqref{eq:double_count_total}, we see that in the part of the  sum containing elements of the first color with a fixed $Y_1 \in \CE_1{Q}$ is
 \[
  \sigma(Q_1 (+)) - \sigma( Q_1 (-) )  - 2 \frac{P^d_{(n,1)}}{2^{n} } 
  \sum\limits_{ Y \; : \; Y_1 \in \CE_1{Q}; \atop Y_1 \text{ is fixed}}    \frac{
   \sgn{{Y_1}} \cdot \sigma\left({Y_1}\right)}{W_1(Y)}.
\]

There are 
$ \binom{d(\sgn {Y_1})}{|{Y_1}|}$  possible sets  ${Y_1}$ for  fixed $\sgn{{Y_1}}$ and $|{Y_1}|.$  These  sets cover set $Q_1 (\sgn Y_1)$ uniformly. That is, we have
\begin{equation}\label{eq:sum_q1}
\sigma(Q_1(+)) - \sigma( Q_1 (-))  - 2\frac{P^d_{(n,1)}}{2^{n} } 
\sum\limits_{\sgn{{Y_1}} \text{ and } |{Y_1}| \atop are fixed} \sgn {Y_1} \cdot 
 \sigma\left(Q_1 (\sgn {Y_1})\right) \frac{M({Y_1})}{W_1({Y_1})},
\end{equation}

where $M({Y_1}) = \frac{|{Y_1}|}{d(\sgn {Y_1})} \binom{d(\sgn {Y_1})}{|{Y_1}|}.$
Using \eqref{eq:double_count_total0} for $r=1,$ we obtain
\[
\sum \frac{P^d_{(n,1)}}{2^{n}}  \frac{\binom{d(\sgn {Y_1})}{|{Y_1}|}}{W_1({Y_1})} = 1,
\]
where the summation is over all admissible $\sgn{{Y_1}}$ and $|{Y_1}|.$
Using this and identity $\sigma(Q_1(+)) = - \sigma(Q_1(-)),$ we have that the expression in \eqref{eq:sum_q1} is equal to $2 \gamma \sigma(Q_1(+)),$ where 
\[
\gamma = \frac{P^d_{(n,1)}}{2^{n}} \sum\limits_0^{d(-)} \left(1 - \frac{j}{d(-)} \right)   \frac{\binom{d(-)}{j}}{\binom{d(+) + j}{j}} + \frac{P^d_{(n,1)}}{2^{n}}\sum\limits_1^{d(+)} \left(1 - \frac{j}{d(+)} \right)   \frac{\binom{d(+)}{j}}{\binom{d(-) + j}{j}}.
\]
After simple transformations, we get
\[
\gamma = \frac{1}{2^n} \left[\sum\limits_0^{d(-)} \left(1 - \frac{j}{d(-)} \right)\binom{n}{d(-) - j}  + \sum\limits_1^{d(+)} \left(1 - \frac{j}{d(+)} \right)  \binom{n}{d(+) - j} \right].
\]
Since $d(-) = n - d(+)$ and $ d(-) - j = d(+) +j,$ we have that $\gamma  < 1.$
As for a lower bound on $\gamma,$ after simple transformations, we have
\[
2^n \gamma = \frac{n}{d(-)} \sum\limits_2^{d(-)} \binom{n-1}{d(-) - j} + \binom{n-1}{d(-) -1} 
+ \binom{n-1}{d(-)} +
\frac{n}{d(+)}  \sum\limits_1^{d(+)} \binom{n-1}{d(-) + j}.
\]
Clearly, the right-hand side here is strictly bigger than $2^{n-1}$ whenever $d(+) \in [n-1].$
Therefore, $2\gamma \in (1,2).$ This completes the proof. 
\end{proof}
The following statement is a colorful version of \Href{Theorem}{thm:Maurey_for_centroid}.
\begin{cor}\label{cor:half_points_approx} Under the above conditions, let additionally 
$S = \bigcup\limits_{i \in [r]} Z_i$ be  a subset of points in a Banach space $X,$ and $D = \max\limits_{i} \diam Z_i.$ Then there is  a partition  $Q^0, Q^1$  of $S$ with $|Q^0_j|= \lfloor \frac{n}{2} \rfloor$ and  $|Q^1_j|= \lceil \frac{n}{2} \rceil$    for every $j \in [r]$ such that
\[ 
 \norm{\centroid{Q^0} - \centroid{S}} \le  \norm{\centroid{Q^1} - \centroid{S}} 
\le C(X)  \left\lceil\frac{n}{2} \right\rceil ^w r^w \diam S
\] 
\end{cor}
\begin{proof}
The first inequality follows from identity 
$$
d\left(\centroid{Q} - \centroid{S} \right) + (n-d)\left(\centroid{S \setminus Q} - \centroid{S} \right) = 0,
$$ 
where $Q \in \binom{S}{d/r}.$ 

By \Href{Lemma}{le:jensen_ineq} for $d = \left\lceil\frac{n}{2} \right\rceil$ and sets $Z_i - \centroid{Z_i}$ and by the definition of type $p,$  we see that there exists $Q^1$ such that
\[
\norm{\sigma(Q^1) - d \centroid{S}} \le \gamma \norm{\sigma(Q^1) - d \centroid{S}}  \le T_p(X) (nr)^{1/p} D.
\]
Dividing the last inequality by $d = \left\lceil\frac{n}{2} \right\rceil,$ we get the needed inequality. 
Clearly, $C(X)$ can be chosen to be $2^{1/p} T_p(X).$

\end{proof}

\section{Proofs of the main results}\label{sec:proofs_main_result}
\begin{proof}[Proof of \Href{Theorem}{thm:col_nodim_Tv}]
We build an incomplete binary tree. Its root is $S$ and its vertices are subsets of $S$. The children of $S$ are $Q^0,  Q^1$ from  \Href{Corollary}{cor:half_points_approx}, the children of $Q^0$ resp. $Q^1$ are $Q^{00}, Q^{01}$ and $Q^{10}, Q^{11}$ obtained again by applying Corollary~\ref{cor:half_points_approx} to $Q^0$ and $Q^1$. 

We split the resulting sets into two parts of as equal sizes as possible the same way, and repeat. We stop when the set $Q^{\delta_1\ldots\delta_h}$ contains exactly one element from each color class. In the end we have sets $S_1,\ldots, S_r$ at the leaves. They form a partition of $S$ with $|S_i\cap Z_j|=1$ for every $j \in [r]$ and  $i \in [k]$. We have to estimate $\norm{\centroid{S_i} - \centroid{S}}$. Let $S,Q^{\delta_1},\ldots,Q^{\delta_1\ldots\delta_h},S_i$ be the sets in the tree on the path from the root to $S_i$. Using the Corollary gives
\begin{eqnarray*}
\norm{\centroid{S} - \centroid{S_i}} \le 
\norm{\centroid{S} - \centroid{Q^{\delta_1}}}  + 
\norm{\centroid{Q^{\delta_1}} - \centroid{Q^{\delta_1 \delta_2}}} 
+ \cdots + \norm{\centroid{Q^{\delta_1\ldots\delta_h}} - \centroid{S_i}} \\
\le C(X) \sum\limits_0^{\infty}  \left( 2^i\right)^w  r^{w} D \le
\frac{C(X)}{1 - 2^\omega} r^w  D.
\end{eqnarray*}

And the constant can be chosen to be $\frac{2^{1/p}}{1 - 2^w} T_p(X).$
\end{proof}

As in  \cite{adiprasito2019theorems}, \Href{Theorem}{thm:col_nodim_Tv} implies the following   statement.
\begin{cor}
 Given a set $P$ of $n$ points in a Banach space $X$ and
an integer $k \in [n],$ there exists a point $q$ and a
partition of $P$ into $k$ sets $P_1, \dots , P_k$ such that
\[
\dist (q, \conv P_i)  \le C(X) \left(\frac{k}{n} \right)^w \diam P 
\]
for every $i \in [k].$
 \end{cor}
\begin{proof}
 Write $|P| = n = kr + s$ with $k \in \mathbb{N}$ so that $0 \le s \le k - 1.$ Then delete $s$ elements from $P$ and split the remaining set into sets (colors) 
 $C_1 , \dots , C_r,$  each of size $k.$ Apply
the colored version and add back the deleted elements(anywhere you like). The outcome is the required partition.

\end{proof}

\begin{proof}[Proof of \Href{Theorem}{th:selection}] This is a combination of \Href{Lemma}{le:CalCar} and the no-dimension Tverberg theorem, like in \cite{barany1982generalization}. We assume that  $n=kr+s$ with $0\le s\le r-1$ ($k$ an integer) and set $\gamma= C(X) r^{w} D,$ where $C(X)$ is a constant that appears in \Href{Theorem}{thm:col_nodim_Tv}. The no-dimension Tverberg theorem implies that $P$ has a partition $\{P_1,\ldots,P_k\}$ such that $\conv P_i$ intersects the ball $B\left(q, \gamma\right)$ for every $i \in [k]$ where $q \in X$ is a suitable point.

Next choose a sequence $1\le j_1\le j_2 \le \ldots \le j_r \le k$ (repetitions allowed) and apply \Href{Lemma}{le:CalCar} to the sets $P_{j_1},\ldots,P_{j_r}$, where we have to set $\eta= \gamma$.
If at this step we have chosen some points several times, we add other arbitrary chosen points of the set $P_{j}$ such that we use the number of appearances of $j$ in $1\le j_1\le j_2 \le \ldots \le j_r \le k$  elements of $P_j$ for each $j \in [r].$
 This gives a transversal $T_{j_1\ldots j_r}$ of $P_{j_1},\ldots,P_{j_r}$ whose convex hull intersects the ball
\[
B\left(q, \gamma + \eta \right).
\]
 So the convex hull of all of these transversals intersects $B\left(q, \gamma+\eta \right)$. They are all distinct $r$-element subsets of $P$ and their number is
\[
{k+r-1 \choose r}={\frac {n-s}r +r-1 \choose r} \ge r^{-r}{n \choose r},
\]
\end{proof}

\begin{proof}[Proof of \Href{Theorem}{th:epsnet}]
The proof is an algorithm that goes along the same lines as in the original weak $\eps$-net theorem \cite{alon1992point}. Set $F:=\emptyset$ and let $\mathcal{H}$ be the family of all $r$-tuples of $P$. On each iteration we will add a point to $F$ and remove $r$-tuples from $\mathcal{H}$.

If there is $Y \subset P$ with $\left(F + B(0, C(X) r^{w} D)\right)\cap \conv Y = \emptyset$, then apply \Href{Theorem}{th:selection} to that $Y$ resulting in a point $q\in X$ such that the convex hull of at least
\[
\frac 1{r^r}{\eps n \choose r}
\]
$r$-tuples from $Y$ intersect $B\left(q, C(X) r^{w} D \right)$. Add the point $q$ to $F$ and delete all $r$-tuples $Q\subset Y$ from $\mathcal{H}$ whose convex hull intersects $B\left(q, C(X) r^{w} D \right)$. On each iteration the size of $F$ increases by one, and at least $r^{-r}{\eps n \choose r}$ $r$-tuples are deleted from $\mathcal{H}.$ So after
\[
\frac {{n \choose r}}{\frac 1{r^r}{\eps n \choose r}} \le \frac {r^r}{\eps^r}
\]
iterations the algorithm terminates as there can't be any further $Y \subset P$ of size $\eps n$ with  $\left(F+ B\left(q, C(X) r^{w} D \right) \right)\cap \conv Y = \emptyset$. Consequently the size of $F$ is at most $r^r{\eps^{-r}}$.
\end{proof}

{\bf Acknowledgements.}  
The author wish to thank Imre B{\'a}r{\'a}ny for bringing the problem to my attention.

\end{document}